\documentclass[12pt]{amsart}

\usepackage{amsaddr}
\usepackage{amsmath}
\usepackage{amssymb}
\usepackage{amscd}
\usepackage{latexsym}
\usepackage{amsthm}
\usepackage{mathrsfs}
\usepackage{color}
\usepackage[backend=bibtex,style=alphabetic,isbn=false,url=false,maxnames=5,firstinits=true]{biblatex}
\usepackage{amsfonts}
\usepackage{enumitem}
\usepackage{array}
\usepackage{setspace}
\usepackage{mathtools}
\usepackage{tikz}
\usepackage{tikz-cd}
\usetikzlibrary{arrows}
\usetikzlibrary{knots}
\renewbibmacro{in:}{}

\usepackage{stmaryrd,graphicx}
\newtheorem{thm}{Theorem}[section]
\newtheorem*{thm*}{Theorem}
\newtheorem{cor}[thm]{Corollary}
\newtheorem*{cor*}{Corollary}
\newtheorem{lem}[thm]{Lemma}
\newtheorem*{lem*}{Lemma}
\newtheorem{prop}[thm]{Proposition}
\newtheorem*{prop*}{Proposition}

\theoremstyle{definition}
\newtheorem{defn}{Definition}[section]
\newtheorem*{defn*}{Definition}

\theoremstyle{remark}

\newtheorem*{rem*}{Remark}

\newtheorem*{problem*}{Problem}

\newcommand{\QQ}{\mathbb Q}
\newcommand{\RR}{\mathbb R}

\newcommand{\cC}{\mathcal C}

\newcommand{\Z}{\mathbb Z}

\newcommand{\cP}{\mathcal P}

\newcommand{\ev}{\mathrm{ev}}

\newcommand{\HH}{\mathrm {HH}}
\newcommand{\HHH}{\mathrm {HHH}}

\newcommand{\eps}{\epsilon}

\newcommand{\Hom}{\mathrm{Hom}}

\DeclareMathOperator{\Ext}{Ext}

\DeclareMathOperator{\area}{area}

\newcommand{\Bim}{\mathrm{Bim}}
\newcommand{\Br}{\mathrm{Br}}

\title{Homology of torus knots}
\author{Anton Mellit}
\email{mellit@gmail.com}
\address{Institute of Science and Technology, \\
Am Campus 1, 3400 Klosterneuburg, Austria}

\begin{document}
\onehalfspacing

\begin{abstract}
Using the method of Elias-Hogancamp and combinatorics of toric braids
we give an explicit formula for the triply graded Khovanov-Rozansky homology of an arbitrary torus knot, thereby proving some of the conjectures of Aganagic-Shakirov, Cherednik, Gorsky-Negut and Oblomkov-Rasmussen-Shende.
\end{abstract}

\maketitle

\section{Elias-Hogancamp recursions}

\subsection{Khovanov-Rozansky homology}
We begin by recalling the construction of Khovanov-Rozansky homology using Soergel bimodules from \cite{khovanov2007triply}. Our notations are close to the one of \cite{elias2016computation}, except that $q$ and $t$ are interchanged and the sign of $a$ is flipped. Let $R_n$ denote the ring of polynomials in $n$ variables,
\[
R_n = \QQ[x_1, x_2, \ldots,x_n].
\]
We will omit $n$ when it is clear. Let $R^{S}$ be the subring of symmetric polynomials. Denote by $\Bim_n$ the category of graded modules over $R\otimes_{R^{S}} R$ which are free both as left and right $R$-modules. These modules will be called \emph{bimodules}. We denote the functor of shifting the degrees by $t$ so that if any expression $x$ has degree $i$ then $tx$ has degree $i+1$. The dg-category of bounded complexes of bimodules is denoted $\cC(\Bim_n)$. If $x$ is a complex, then $-(qt)^{-\frac12}x$ stands for a complex whose homological\footnote{What the authors of \cite{elias2016computation} call \emph{homological} degree should be called \emph{cohomological} degree} degrees are increased by $1$. $\HH$ stands for the Hochshild cohomology functor $M\to \Ext^\bullet_{R\otimes_{R^S} R}(R, M)$. When applied to an object of $\Bim_n$ it produces a graded object in $\Bim_n$, thus a doubly graded bimodule. The degree shifts with respect to this new grading will be denoted by $-at$. When we apply $\HH$ to a complex $x\in \cC(\Bim_n)$, we obtain a complex of doubly graded vector spaces. In the dg-category of vector spaces every object is equivalent to its homology, so the corresponding triply graded vector space will be denoted by $\HHH(x)$. Counting the dimensions of graded pieces produces an infinite series in $a$, $t$, $\sqrt{qt}$ and their inverses, which we will also denote by $\HHH(x)$.

We will denote the left action of the generators of $R$ on a bimodule by $x_1, x_2, \ldots$ and the right action by $x_1', x_2', \ldots$.

To get accustomed to the notations we evaluate $\HH(R_n)$ and $\HHH(R_n)$. A free resolution of $R$ over $R\otimes_{R^{S}} R$ can be given by a complex of the form 
\[
\cdots \to t V\otimes R\otimes_{R^{S}} R \to R\otimes_{R^{S}} R,
\]
where $V=\bigoplus_{i=1}^n \QQ e_i / \sum_{i=1}^n e_i$ is a vector space of dimension $n-1$. The last differential sends $e_i$ to $x_i - x_i'$. This increases the degree by $1$, therefore we need the factor of $t$. The $k$-th term in the complex is $t^k \Lambda^k V R\otimes_{R^{S}} R$. Applying $\Hom(-,R)$ produces a complex whose differentials are $0$, so we obtain
\[
\HH(R_n) = \bigoplus_{i=0}^{n-1} \binom{n-1}{i} t^{-k} (-a t)^k R_n = (1-a)^{n-1} R_n.
\]
Therefore
\[
\HHH(R_n) = \frac{(1-a)^{n-1}}{(1-t)^n}.
\]

Denote by $B_i$ the \emph{Bott-Samelson bimodule}
\[
B_i = R\otimes_{R^{\sigma_i}} R,
\]
where $R^{\sigma_i}$ denotes the ring of polynomials invariant under $\sigma_i$, which interchanges the $i$-th and $i+1$-st variables. Then form the \emph{Rouquier complex}
\footnote{Our complexes are slightly different from the standard ones. To obtain the standard complexes simply multiply ours by $q^{\frac12}$. This change is convenient because we never have to deal with $q^{\frac12}$, $t^{\frac12}$ separately, they always come in pairs as in $(qt)^\frac12$, those monomials that have $(qt)^\frac12$ correspond to odd homology, while those monomials without it correspond to even homology.}
\[
T_i = \left(B_i \to \underline{R}\right).
\]
Here when writing a complex we underline the part in homological degree $0$. In the tensor category $\cC(\Bim_n)$ this complex turns out to be invertible with inverse given by
\[
T_i^{-1} = \left(\underline{R}\to t^{-1} B_i\right).
\]
The map $R\to t^{-1}B_i$ corresponds to a degree $1$ map $R\to B_i$, which is given by sending $1\in R$ to $x_i-x_{i+1}'$. This is indeed a map of bimodules because $R$ is generated over $R^{\sigma_i}$ by $x_{i}$, so it is enough to check
\[
x_i(x_i - x_{i+1}') - x_i'(x_i - x_{i+1}')=(x_i-x_i')(x_i - x_{i+1}')=0 \quad\text{in $B_i$.}
\]
The tensor product of $T_i$ and $T_i^{-1}$ looks as follows:
\[
B_i \to \underline{R\oplus t^{-1}B_i\otimes_{R_n} B_i} \to t^{-1} B_i.
\]
It is easy to see that $B_i\otimes_R B_i\cong B_i\oplus t B_i$ and the complex above splits into $R$ and two homotopically trivial complexes $B_i\to B_i$ and $t^{-1} B_i \to t^{-1} B_i$.

It turns out that this construction gives a homomorphism from the Artin braid group on $n$ strands, denoted $\Br_n$, to $\cC(\Bim_n)$ by sending the $i$-th standard generator of $\Br_n$ to $T_i$. It is convenient to identify a braid with its image in $\cC(\Bim_n)$. Furthermore, $\HHH(\beta)$ is almost an invariant of the link $[\beta]$ obtained by closing $\beta$. The link invariant $\cP(\beta)$ is given by
\[
\cP(\beta) = \left(a q^\frac12 t^{-\frac12}\right)^{\chi(\beta)} \HHH(\beta),\quad \chi(\beta)=\frac{e(\beta)-n+c(\beta)}2
\]
where $e(\beta)$ is the sum of the exponents of the braid group generators in $\beta$, $c(\beta)$ is the number of components of the link $[\beta]$. Note that $\chi(\beta)$ is always an integer.

 For instance, for $n=2$ we have
\[
\HH^1(T_1) = (t^{-1} R \xrightarrow{\sim} \underline{t^{-1} R}) \cong 0,\quad
\HH^0(T_1) = (t R \xrightarrow{x_1-x_2} \underline{R}),
\]
\[
\HH^1(T_1^{-1}) = (t^{-1} R \xrightarrow{x_1-x_2} \underline{t^{-2} R}),\quad
\HH^0(T_1^{-1}) = (R \xrightarrow{\sim} \underline{R})\cong 0,
\]
so
\[
\HHH(T_1)=\frac{1}{1-t}=\HHH(R_1),\quad \HHH(T_1^{-1})=\frac{aq^{\frac12} t^{-\frac12}}{1-t},
\]
\[
\cP(T_1)=\cP(T_1^{-1})=\frac{1}{1-t},
\]
as expected.


\subsection{Categorified symmetrizer} In \cite{hogancamp2015categorified} a categorification of the Jones-Wenzl idempotents is constructed. Following \cite{elias2016computation} we prefer to deal with its 'multiple'. For $1\leq l\leq n$ the complexes $K_l\in\cC(\Bim_n)$ are defined recursively by 
\[
K_1 = R, \qquad K_{l+1} = (t K_l \xrightarrow{f_l} \underline{ q^l T_l T_{l-1} \cdots T_1 T_1 \cdots T_{l-1} T_l K_l}),
\]
for certain morphisms $f_l$. Then it is shown that
\begin{enumerate}
\item $K_l T_i \cong T_i K_l \cong K_l$ for all $i<l$,
\item $\HHH((1\sqcup \beta)K_l) = (q^{l-1}-a) \HHH(\beta K_{l-1})$ for any braid $\beta\in\Br_{n-1}$.
\end{enumerate}
Note that the factor $q^l$ takes care of the difference between our Rouquier complexes and the standard ones. So our $K_l$ agree with the ones given in \cite{elias2016computation}. This complexes are used as follows. Let $F_\ev$ be the space of series that depend only on integer powers of $q$, $t$. These correspond to complexes with only even homology.
\begin{lem}\label{lem:reduction1}
Suppose $\HHH(\beta K_{l+1}), \HHH(\beta K_l)\in F_\ev$  for some braid $\beta\in\Br_{n}$. Then we have
\[
\HHH(\beta T_l T_{l-1} \cdots T_1 T_1 \cdots T_{l-1} T_l K_l) = q^{-l}\left(t\, \HHH(\beta K_l) + \HHH(\beta K_{l+1})\right).
\]
\end{lem}

In fact, this can be strengthened as follows. For each integer $m$ denote by $F_{\geq m}$ the space of power series containing only terms $q^i t^j a^m$ with $j\geq m$.
\begin{lem}\label{lem:reduction2}
Suppose $\HHH(\beta K_{l+1})\in F_{\geq m} + F_\ev$ and $\HHH(\beta K_l)\in F_{\geq m-1} + F_\ev$. Then we have
\[
\HHH(\beta T_l T_{l-1} \cdots T_1 T_1 \cdots T_{l-1} T_l K_l) = q^{-l}\left(t\, \HHH(\beta K_l) + \HHH(\beta K_{l+1})\right) \pmod {F_{\geq m}}.
\]
\end{lem}

\section{The main construction}
\subsection{From coloring to a complex}
Next we explain how to organize the identities into a recursion to compute $\HHH$ of torus knots and links. The state of the recursion is precisely the \emph{admissible coloring} of a line from \cite{mellit2016toric}:
\begin{defn}\label{defn:colorings}
Let $h\in\RR, s\in\RR_{>0}$ be such that $s\notin\QQ$ and the line $l_h:=\{x,y:y=s x + h\}$ does not pass through a lattice point. An \emph{admissible coloring} of $l_h$ is a finite family of non-overlapping intervals $[a_i, b_i]\subset l_h$ with $x(a_i)<x(b_i)$ such that each $a_i$ belongs to a vertical grid line, and each $b_i$ belongs to a horizontal grid line, i.e. $x(a_i)\in\Z$ and $y(b_i)\in\Z$ for all $i$.
\end{defn}

To an admissible coloring we associate a complex in $\cC(\Bim_{N(c)})$ as follows. Denote the number of the intervals of a coloring $c$ by $k(c)$ and the endpoints of the intervals by $a_i(c), b_i(c)$. The number of strands $N(c)$ will be the sum of $k(c)$ and the number of times an interval crosses a vertical lattice, i.e.
\[
N(c) = k(c) + \sum_{i=1}^{k(c)} \lfloor y(b_i(c))-y(a_i(c)) \rfloor.
\]
In the beginning we put the symmetrizer $K_k$. Projecting the intervals to the unit square we obtain a collection of non-intersecting intervals in the unit square with endpoints on the sides. We put the square below and to the right of the symmetrizer and explain how to connect the endpoints of the intervals to form a braid. The endpoints of our intervals on the top side of the square that are \emph{not} $b_i$ simply extend vertically to the top of the picture. The endpoints that are $b_i$ connect to the symmetrizer passing below the other strands. The endpoints on the bottom side extend vertically to the bottom of the picture. The endpoints on the left side that are $a_i$ extend to the left and connect to the bottom of the picture. The remaining endpoints on the left side connect to the corresponding endpoints on the right side passing below the other strands. Note that all crossings in the resulting braid are positive.

For example, consider the coloring on Figure \ref{fig:example coloring}. It has $3$ intervals. The resulting complex is shown to the right. On the picture we do not show the strands that go from the left side of the square to the right side. They are below the square, so we may think that the square is an opaque piece of paper, and the strands wrap around it.

Our construction may look strange, but in fact it is obtained by drawing the strands on the punctured torus embedded in $\RR^3$, placing the symmetrizer in the puncture and connecting the ends. Then the resulting configuration in $\RR^3$ is projected to the plane.

\begin{figure}
\def\eps{0.1}
\def\s{(4/7)}
\def\t{(1/(1+\s))}
\def\m{7}
\def\n{4}
\def\h{(\n-(\m-\eps)*\s)}
\begin{tikzpicture}[scale=0.7]
\draw[gray, step=1, thin] (0,0) grid (\m,\n);
\draw[gray,thin,->] (0, 0) -- (\m+0.5, 0) node[black,right] {$x$};
\draw[gray,thin,->] (0, 0) -- (0, \n+0.5) node[black,above] {$y$};
\draw (0, {\h}) -- ({(1-\h) / \s},1);
\draw (2, {\h+2*\s}) -- ({(3-\h) / \s},3);
\draw (6, {\h+6*\s}) -- ({(4-\h) / \s},4);
\end{tikzpicture}
\qquad
\def\eps{0.1}
\def\s{(4/7)}
\def\t{(1/(1+\s))}
\def\m{7}
\def\n{4}
\def\h{(\n-(\m-\eps)*\s)}
\begin{tikzpicture}[scale=2]
\draw (-1,1.9) rectangle (-0.5,1.7);
\draw (-0.75,1.8) node {$K_3$};
\draw (0,0) rectangle (1,1);

\draw (0, {\h})node[draw,circle, inner sep=1pt,fill] {} -- (1, {\h+\s});
\draw (0, {\h+\s}) -- ({1/\s*(1-\h)-1}, 1) node[draw,circle, inner sep=1pt,fill] {};
\draw (0, {\h+2*\s-1})node[draw,circle, inner sep=1pt,fill] {} -- (1, {\h+3*\s-1});
\draw (0, {\h+3*\s-1}) -- ({-3 +(2-\h)/\s},1);
\draw (({-3 +(2-\h)/\s},0) -- (1, {\h+4*\s-2});
\draw (0,{\h+4*\s-2}) --(1,{\h+5*\s-2});
\draw (0,{\h+5*\s-2}) -- ({-5 +(3-\h)/\s},1)node[draw,circle, inner sep=1pt,fill] {};
\draw (0,{\h+6*\s-3})node[draw,circle, inner sep=1pt,fill] {} -- ({-6 +(4-\h)/\s},1)node[draw,circle, inner sep=1pt,fill] {};
\begin{knot}[clip width=4]
\strand (-0.75,1.7) to [out=-90,in=90,looseness=0.6] ({{1/\s*(1-\h)-1}}, 1);
\strand (-0.9,1.7) to [out=-90,in=90,looseness=0.6] ({{-5 +(3-\h)/\s}},1);
\strand (-0.6,1.7) to [out=-90,in=90,looseness=0.6] ({{-6 +(4-\h)/\s}},1);
\strand ({-3 +(2-\h)/\s}, 1) -- ({-3 +(2-\h)/\s}, 2);
\strand ({-3 +(2-\h)/\s}, 0) -- ({-3 +(2-\h)/\s}, -0.1);
\strand (-0.75,1.9) -- (-0.75,2);
\strand (-0.9,1.9) -- (-0.9,2);
\strand (-0.6,1.9) -- (-0.6,2);
\strand (-0.6,-0.1) to [out=90,in=-180,looseness=0.6] (0, {\h});
\strand (-0.75,-0.1) to [out=90,in=-180,looseness=0.6] (0, {\h+2*\s-1});
\strand (-0.9,-0.1) to [out=90,in=-180,looseness=0.6] (0,{\h+6*\s-3});
\flipcrossings{1,2}
\end{knot}
\end{tikzpicture}
\caption{A coloring and the corresponding complex}
\label{fig:example coloring}
\end{figure}
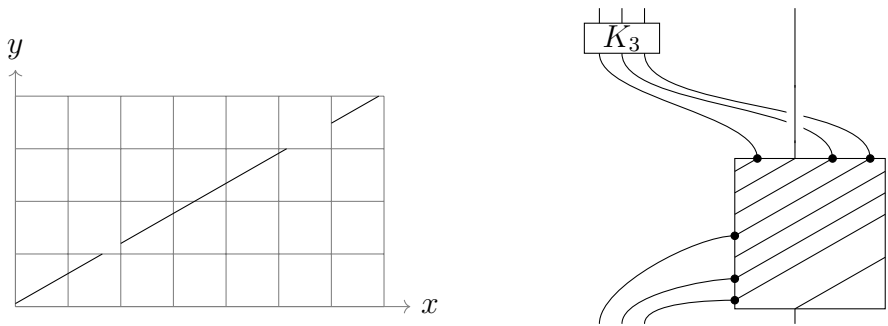

\subsection{Transformation rules}
Next we study what happens when we move the line $l_h$ upwards increasing $h$. The possibilities are precisely the ones that appeared in \cite{mellit2016toric}. In the following table the colorings are denoted by $c$, $c'$, $c''$. We write $\HHH(c)$, $\HHH(c')$, $\HHH(c'')$ for the $\HHH$ of the corresponding complexes. The resulting transformation rules are given on Figure \ref{fig:transformation}. On the diagrams, the intervals of a coloring are represented by the intersection of the shaded region with the corresponding dashed line. For instance, if we start with our example on Figure \ref{fig:example coloring}, we will have to apply rule D) first. So $\HHH$ of the original complex is equal to the $\HHH$ of the complex on Figure \ref{fig:example step2}. Then we are in the situation to apply rule BE), so our computation is reduced to the computation of $\HHH$ of two complexes.

The rule BE) is conditional in the sense that it can be applied only if we know that $\HHH(c')$ and $\HHH(c'')$ are in $F_\ev$ (see Lemma \ref{lem:reduction1}). It is clear however that starting from any coloring our process terminates, so all the intermediate values of $\HHH$ will be computed, and all will be in $F_\ev$.
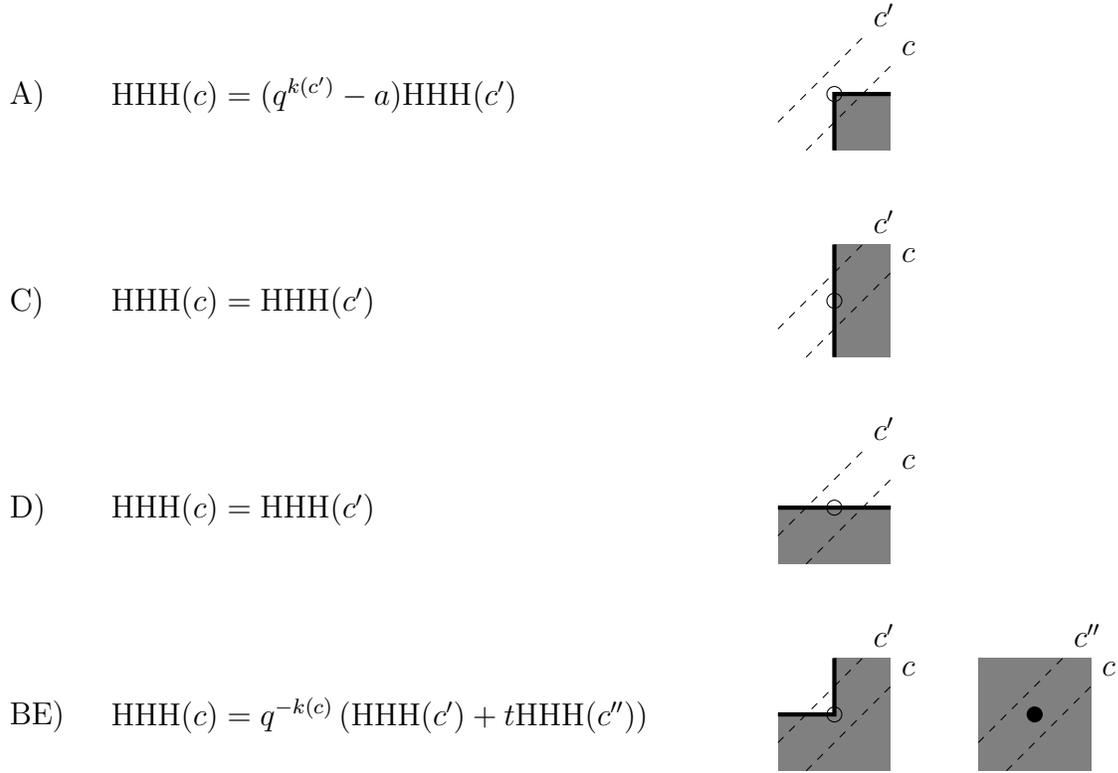
\begin{figure}
\begin{tabular}{m{1cm} m{8cm} m{6cm}}
A) & 
$\HHH(c) = (q^{k(c')}-a)\HHH(c')$ &
\begin{tikzpicture}[scale=0.75]
\fill[gray] (0,-1) rectangle (1,0);
\draw[black,ultra thick](0,-1) -- (0,0) -- (1,0);
\draw (0,0) node[draw,circle, inner sep=2pt] {};
\draw[dashed](-0.5,-1)--(1, 0.5) node[above right] {$c$};
\draw[dashed](-1,-0.5)--(0.5, 1) node[above right] {$c'$};
\draw (-1.5, -1.5) node {};
\draw (1.5, 1.5) node {};
\end{tikzpicture}
\\
C) & 
$\HHH(c) = \HHH(c')$ &
\begin{tikzpicture}[scale=0.75]
\fill[gray] (0,-1) rectangle (1,0);
\fill[gray] (0,0) rectangle (1,1);
\draw[black,ultra thick](0,-1) -- (0,0) -- (0,1);
\draw (0,0) node[draw,circle, inner sep=2pt] {};
\draw[dashed](-0.5,-1)--(1, 0.5) node[above right] {$c$};
\draw[dashed](-1,-0.5)--(0.5, 1) node[above right] {$c'$};
\draw (-1.5, -1.5) node {};
\draw (1.5, 1.5) node {};
\end{tikzpicture}
\\
D) & 
$\HHH(c) = \HHH(c')$ &
\begin{tikzpicture}[scale=0.75]
\fill[gray] (-1,-1) rectangle (0,0);
\fill[gray] (0,-1) rectangle (1,0);
\draw[black,ultra thick](-1,0) -- (0,0) -- (1,0);
\draw (0,0) node[draw,circle, inner sep=2pt] {};
\draw[dashed](-0.5,-1)--(1, 0.5) node[above right] {$c$};
\draw[dashed](-1,-0.5)--(0.5, 1) node[above right] {$c'$};
\draw (-1.5, -1.5) node {};
\draw (1.5, 1.5) node {};
\end{tikzpicture}
\\
BE) & 
$\HHH(c) = q^{-k(c)}\left(\HHH(c')+t\HHH(c'')\right)$ &
\begin{tikzpicture}[scale=0.75]
\fill[gray] (-1,-1) rectangle (0,0);
\fill[gray] (0,-1) rectangle (1,0);
\fill[gray] (0,0) rectangle (1,1);
\draw[black,ultra thick](-1,0) -- (0,0) -- (0,1);
\draw (0,0) node[draw,circle, inner sep=2pt] {};
\draw[dashed](-0.5,-1)--(1, 0.5) node[above right] {$c$};
\draw[dashed](-1,-0.5)--(0.5, 1) node[above right] {$c'$};
\draw (-1.5, -1.5) node {};
\draw (1.5, 1.5) node {};
\end{tikzpicture}
\begin{tikzpicture}[scale=0.75]
\fill[gray] (-1,-1) rectangle (1,1);
\draw (0,0) node[draw,circle, inner sep=2pt,fill] {};
\draw[dashed](-0.5,-1)--(1, 0.5) node[above right] {$c$};
\draw[dashed](-1,-0.5)--(0.5, 1) node[above right] {$c''$};
\draw (-1.5, -1.5) node {};
\draw (1.5, 1.5) node {};
\end{tikzpicture}
\\
\end{tabular}
\caption{Transformation rules}
\label{fig:transformation}
\end{figure}

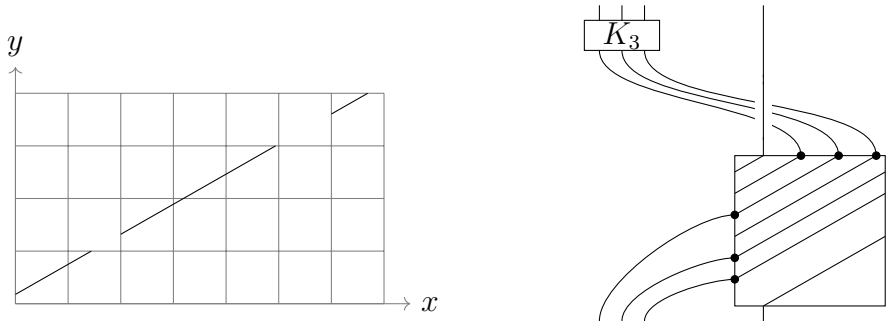
\begin{figure}
\def\eps{0.1}
\def\s{(4/7)}
\def\t{(1/(1+\s))}
\def\m{7}
\def\n{4}
\def\h{(\n-(\m-\eps)*\s+0.12)}
\begin{tikzpicture}[scale=0.7]
\draw[gray, step=1, thin] (0,0) grid (\m,\n);
\draw[gray,thin,->] (0, 0) -- (\m+0.5, 0) node[black,right] {$x$};
\draw[gray,thin,->] (0, 0) -- (0, \n+0.5) node[black,above] {$y$};
\draw (0, {\h}) -- ({(1-\h) / \s},1);
\draw (2, {\h+2*\s}) -- ({(3-\h) / \s},3);
\draw (6, {\h+6*\s}) -- ({(4-\h) / \s},4);
\end{tikzpicture}
\qquad
\def\eps{0.1}
\def\s{(4/7)}
\def\t{(1/(1+\s))}
\def\m{7}
\def\n{4}
\def\h{(\n-(\m-\eps)*\s+0.12)}
\begin{tikzpicture}[scale=2]
\draw (-1,1.9) rectangle (-0.5,1.7);
\draw (-0.75,1.8) node {$K_3$};
\draw (0,0) rectangle (1,1);

\draw (0, {\h})node[draw,circle, inner sep=1pt,fill] {} -- (1, {\h+\s});
\draw (0, {\h+\s}) -- ({1/\s*(1-\h)-1}, 1) node[draw,circle, inner sep=1pt,fill] {};
\draw (0, {\h+2*\s-1})node[draw,circle, inner sep=1pt,fill] {} -- (1, {\h+3*\s-1});
\draw (0, {\h+3*\s-1}) -- ({-3 +(2-\h)/\s},1);
\draw (({-3 +(2-\h)/\s},0) -- (1, {\h+4*\s-2});
\draw (0,{\h+4*\s-2}) --({-4 +(3-\h)/\s},1)node[draw,circle, inner sep=1pt,fill] {};
\draw (0,{\h+6*\s-3})node[draw,circle, inner sep=1pt,fill] {} -- ({-6 +(4-\h)/\s},1)node[draw,circle, inner sep=1pt,fill] {};
\begin{knot}[clip width=4]
\strand (-0.9,1.7) to [out=-90,in=90,looseness=0.6] ({{1/\s*(1-\h)-1}}, 1);
\strand (-0.6,1.7) to [out=-90,in=90,looseness=0.6] ({{-4 +(3-\h)/\s}},1);
\strand (-0.75,1.7) to [out=-90,in=90,looseness=0.6] ({{-6 +(4-\h)/\s}},1);
\strand ({-3 +(2-\h)/\s}, 1) -- ({-3 +(2-\h)/\s}, 2);
\strand ({-3 +(2-\h)/\s}, 0) -- ({-3 +(2-\h)/\s}, -0.1);
\strand (-0.75,1.9) -- (-0.75,2);
\strand (-0.9,1.9) -- (-0.9,2);
\strand (-0.6,1.9) -- (-0.6,2);
\strand (-0.6,-0.1) to [out=90,in=-180,looseness=0.6] (0, {\h});
\strand (-0.75,-0.1) to [out=90,in=-180,looseness=0.6] (0, {\h+2*\s-1});
\strand (-0.9,-0.1) to [out=90,in=-180,looseness=0.6] (0,{\h+6*\s-3});
\flipcrossings{1,2}
\end{knot}
\end{tikzpicture}
\caption{The complex of Figure \ref{fig:example coloring} after rule D)} and before rule BE).
\label{fig:example step2}
\end{figure}

\subsection{Representation of a torus knot}
For any pair of relatively prime $m,n$ the $(m,n)$-torus knot can be obtained by composing $m$ rotations on $n$ strands, a rotation being defined as $T_{n-1} \cdots T_2 T_1$. If we take a single interval from $(\varepsilon,0)$ to $(m+\varepsilon,n)$ for small $\varepsilon>0$, project it to the unit square and proceed as before, extending the strands vertically and connecting the endpoints on the left side of the square to the endpoints on the right side, we obtain a braid whose closure is the $(m,n)$-torus knot. Tilt the interval slightly to obtain an interval $(0,\varepsilon)-(m+\varepsilon,n)$ with a very small $\varepsilon>0$. The result is an admissible coloring, with a single interval, which we denote by $c_{m,n}$. The complex corresponding to $c_{m,n}$ differs from the complex of the $(m,n)$-knot only by having $K_1$ on the leftmost strand, so its $\HHH$ is the same. See Figure \ref{fig:example43} for an example with $(4,3)$ knot.
\begin{figure}
\[
\def\eps{0.1}
\def\s{(3/4)}
\def\t{(1/(1+\s))}
\def\m{4}
\def\n{3}
\def\h{-0.1}
\begin{tikzpicture}[scale=0.5,baseline=30pt]
\draw[gray, step=1, thin] (0,0) grid (\m,\n);
\draw[gray,thin,->] (0, 0) -- (\m+0.5, 0) node[black,right] {$x$};
\draw[gray,thin,->] (0, 0) -- (0, \n+0.5) node[black,above] {$y$};
\draw ({-\h/\s}, 0) -- ({(-\h+3)/\s},3);
\end{tikzpicture}
\quad
\begin{tikzpicture}[scale=2,baseline=30pt]
\draw (0,0) rectangle (1,1);
\draw ({-\h/\s},-0.1)-- +(0,0.1) -- (1,{\h+\s});
\draw (0,{\h+\s}) -- ({(-\h+1)/\s-1},1) -- +(0,0.5);
\draw ({(-\h+1)/\s-1},-0.1)-- +(0,0.1) -- (1,{\h+2*\s-1});
\draw (0,{\h+2*\s-1}) -- ({(-\h+2)/\s-2},1) -- +(0,0.5);
\draw ({(-\h+2)/\s-2},-0.1)-- +(0,0.1) -- (1,{\h+3*\s-2});
\draw (0,{\h+3*\s-2}) -- (1,{\h+4*\s-2});
\draw (0,{\h+4*\s-2}) -- ({-\h/\s},1) -- +(0,0.5);
\end{tikzpicture}
\quad \Rightarrow \quad
\def\s{((3/4)-0.05)}
\def\h{0.1}
\begin{tikzpicture}[scale=0.5,baseline=30pt]
\draw[gray, step=1, thin] (0,0) grid (\m,\n);
\draw[gray,thin,->] (0, 0) -- (\m+0.5, 0) node[black,right] {$x$};
\draw[gray,thin,->] (0, 0) -- (0, \n+0.5) node[black,above] {$y$};
\draw (0, {\h}) -- ({(-\h+3)/\s},3);
\end{tikzpicture}
\quad
\begin{tikzpicture}[scale=2,baseline=30pt]
\draw (-1,1.4) rectangle (-0.5,1.2);
\draw (-0.75,1.3) node {$K_1$};
\draw (0,0) rectangle (1,1);

\draw (-0.75,-0.1) to[out=90,in=-180,looseness=0.5] (0,{\h})node[draw,circle, inner sep=1pt,fill]{} -- (1,{\h+\s});
\draw (0,{\h+\s}) -- ({(-\h+1)/\s-1},1) -- +(0,0.5);
\draw ({(-\h+1)/\s-1},-0.1)-- +(0,0.1) -- (1,{\h+2*\s-1});
\draw (0,{\h+2*\s-1}) -- ({(-\h+2)/\s-2},1) -- +(0,0.5);
\draw ({(-\h+2)/\s-2},-0.1)-- +(0,0.1) -- (1,{\h+3*\s-2});
\draw (0,{\h+3*\s-2}) -- (1,{\h+4*\s-2});
\draw (0,{\h+4*\s-2}) -- ({(-\h+3)/\s-4},1)node[draw,circle, inner sep=1pt,fill]{} to[out=90,in=-90,looseness=0.5] (-0.75,1.2);
\draw (-0.75,1.4)--(-0.75,1.5);
\end{tikzpicture}
\]
\caption{From $(4,3)$-torus knot to our presentation.}
\label{fig:example43}
\end{figure}
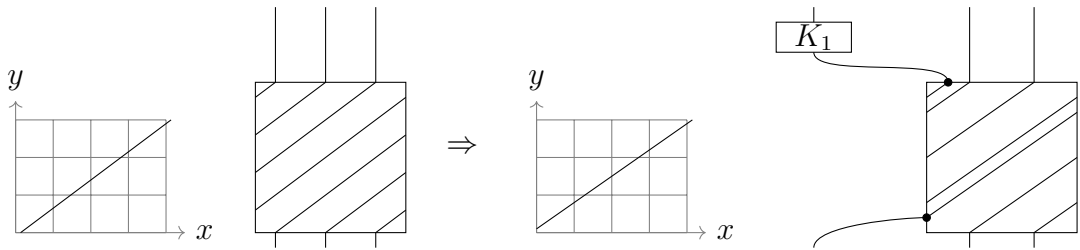

\section{Explicit formulas}
\subsection{Dyck paths}
Let us start with the coloring for the $(m,n)$-knot and move the line upwards. In the process of moving the line each time we have to apply the rule BE) choose one of the two colorings: $c'$ or $c''$ and proceed. Let us shade the region swept by the intervals of the coloring. Its boundary is an $(m,n)$-Dyck path, and all Dyck paths are obtained in this way.
\begin{defn}
An $(m,n)$-Dyck path is a lattice path from $(0,0)$ to $(m,n)$ consisting only of steps $(1,0)$ and $(0,1)$ and staying weakly above the line $y=\frac{n}{m}x$, which is called the \emph{diagonal}.
\end{defn}
The observation that Dyck paths are in one-to-one correspondence with sequences of choices for the rule BE) almost directly translates into the following result, which appears as a conjecture in \cite{oblomkov2012hilbert} and \cite{gorsky2015refined}.
\begin{thm}\label{thm: dyck paths}
The triply graded homology invariant $\cP_{m,n}=\cP(c_{m,n})$ of the $(m,n)$-torus knot is given by
\[
\left(a (qt)^{-\frac12}\right)^{\frac{(m-1)(n-1)}2}\frac{1}{1-t} \sum_{\pi} t^{\area(\pi)} q^{h_+(\pi)} \prod_{v\in v^*(\pi)} (1-a q^{-k(v)}),
\]
where for each $(m,n)$-Dyck path $\pi$ we use the following notations: 
\begin{itemize} 
\item $\area(\pi)$ is the number of $1\times 1$ lattice squares contained between the path and the diagonal,
\item $v^*(\pi)$ is the set of outer vertices of $\pi$ without the vertex most distant from the diagonal,
\item $k(v)$ for each $v\in v^*(\pi)$ is the number of vertical steps ($=$number of horizontal steps) of $\pi$ intersected by the line parallel to the diagonal passing through $v$.
\item $h_+(\pi)$ is the number of pairs ($s_1$, $s_2$) where $s_1$ resp. $s_2$ is the horizontal resp. vertical step of $\pi$, $s_1$ is to the left of $s_2$ and there exists a line parallel to the diagonal intersecting both $s_1$ and $s_2$.
\end{itemize}
\end{thm}
\begin{proof}
Let $\beta$ be the composition of $m$ rotations on $n$ strands. We have $e(\beta)=m(n-1)$, $c(\beta)=1$, therefore $\chi(\beta)=\frac{(m-1)(n-1)}2$. So it is enough to show that
\[
\HHH(\beta) = \frac{1}{1-t} \sum_{\pi} t^{\area(\pi)} q^{h_+(\pi)-\frac{(m-1)(n-1)}2} \prod_{v\in v^*(\pi)} (1-a q^{-k(v)}),
\]
which we rewrite as follows:
\[
\HHH(\beta) = \frac{1}{1-t} \sum_{\pi} t^{\area(\pi)} q^{h_+(\pi)-\frac{(m-1)(n-1)}2-\sum_{v\in v^*(\pi)} k(v)} \prod_{v\in v^*(\pi)} (q^{k(v)}-a).
\]
Each Dyck path corresponds to a sequence of choices between $c'$ and $c''$ in the rule BE). We claim that each summand in the sum corresponds to the contribution calculated from the transformation rules. The power of $t$ obtained from choosing $c''$ matches $\area(\pi)$. Factors $q^{k(v)}-a$ correspond to the applications of rule A). Thus it remains to match the power of $q$ for each Dyck path. This amounts to showing
\[
h_+(\pi)-\frac{(m-1)(n-1)}2-\sum_{v\in v^*(\pi)} k(v) = -\sum_{p\in i(\pi)\cup v_*(\pi)} k(p),
\]
where $i(p)$ is the set of lattice points between $\pi$ and the diagonal, $v_*(\pi)$ is the set of internal vertices of $\pi$, and $k(p)$ for any lattice point is defined in the same way as $k(v)$ was. A proof of this identity is given below.

When we say ``a line through $p$'' we will always mean the line parallel to the diagonal passing through $p$. Let $o(\pi)$ be the set of pairs $(s_1, s_2)$ where $s_1$ resp. $s_2$ is the horizontal resp. vertical step of $\pi$ and $s_1$ is to the left of $s_2$. Note that such pairs correspond to $1\times 1$ grid squares which are outside $\pi$ but inside that right-angled triangle with vertices $(0,0)$, $(m,n)$, $(0,n)$. Therefore $|i(\pi)|+|o(\pi)|$ is the total number of squares inside the triangle, i.e.
\[
 |i(\pi)|+|o(\pi)| = \frac{(m-1)(n-1)}2.
\]
Next look at the complement of the set of pairs contributing to $h_+(\pi)$ in $o(\pi)$. This is the set of pairs $(s_1, s_2)$ such that either $s_1$ is further away from the diagonal than $s_2$ is or vice versa. In the first case the rightmost end of $s_1$ is further than $s_2$, so there is a unique point $p\in i(\pi)$ with the same $x$-coordinate as the rightmost end of $s_1$ and such that the line through it intersects $s_2$. Vice versa, starting from a point $p \in i(\pi)$ and a vertical step $s_2$ to the right of it such that the line through $p$ intersects $s_2$, we recover $s_1$ uniquely. Similarly, starting from a pair $(s_1, s_2)$ such that $s_2$ is further away from the diagonal we find a point $p\in i(\pi)$ with the same $y$-coordinate as the lower end of $s_2$ and such that line through $p$ intersects $s_1$, and such point allows to reconstruct $s_2$. So we obtain
\[
|o(\pi)|-h_+(\pi) = \sum_{p\in i(\pi)} (k(p)-1),
\]
because for $p\in i(p)$ the number of horizontal steps to the left of $p$ which intersect the line through $p$ is $1$ less than the number of vertical such steps. So we have established the following:
\[
\frac{(m-1)(n-1)}2 = |o(\pi)| + |i(\pi)| = h_+(\pi) + \sum_{p\in i(\pi)} k(p).
\]
It remains to show that 
\[
\sum_{p\in v_*(\pi)} k(p) = \sum_{p\in v^*(\pi)} k(p).
\]
Let us make a list of symbols $+/-$ as follows: order the set of all vertices of $\pi$ by decreasing their distance to the diagonal and put $+$ resp. $-$ for each external resp. internal vertex. Let the total number of $+$ be $r$, so that the number of $-$ is $r-1$. Then the number on the left hand side is the number of pairs $+$, $-$ appearing in the list in this order minus the number of pairs $-$, $-$ minus $r-1$ (the last minus is because we have to subtract $1$ for each $-$). The number on the right hand side is the number of pairs $+$, $+$ minus the number of pairs $-$, $+$, again in this order. Then the difference between the left hand side and the right hand side is
\[
r(r-1) - (r-1) - \binom{r}{2} - \binom{r-1}{2} =0.
\]
\end{proof}

\subsection{Toric braids}
In \cite{mellit2016toric} an action of a (positive part of) the toric braid group was constructed on the vector space $V_k$, which is the space of functions which are symmetric functions in $X$ and polynomials in $y_1, y_2, \ldots, y_k$. The colorings were interpreted as toric braids, and then evaluated in the representation. Such evaluations were shown to satisfy recursions under the transformations A), C), D), BE) and so were connected to Dyck paths. Now we will directly connect them to $\HHH$. For any coloring $c$ denote by $B_c$ the corresponding toric braids and by $I(c)$ the evaluation
\[
\left(d_-^{k(c)} B_c d_+^{k(c)} (1)\right)[a-1].
\]
Recall that the result of $d_+^{k(c)}$ is a vector in $V_{k(c)}$. Then we apply the braid, and then the operators $d_-$ which bring us to the space $V_0$, which is simply the space of symmetric functions. Finally, we plethystically substitute $[a-1]$. It turns out that this substitution simplifies matters a lot:

\begin{prop}\label{prop: simplification}
For any vector $F\in V_k$ of the form $F=F' \prod_{i=1}^k y_i$ denote by $I(v)$ the product
\[
I(F)=F'[a-q^k; 1, q,\ldots,q^{k-1}] \prod_{i=0}^k (q^i-a),
\]
where in the right hand side we perform plethystic substitution $X=a-q^k$, $y_i=q^{i-1}$ for $i=1,\ldots,k$. Then we have
\[
I(d_- F) = F,\qquad I(d_+ F) = (a-q^k) I(F).
\]
\end{prop}

In particular, we have $I(c)=I(B_c d_+^{k(c)}(1))$ and the recursions satisfied by $I(c)$ are similar to the ones for $\HHH(c)$. To be precise, we have:
\begin{thm}\label{thm: daha}
For any coloring $c$ let $\widetilde{\HHH}(c)$ be the complex constructed like the one in the main construction, but using the standard Rouquier complexes, which are $q^\frac12 T_i$ in our notation. Then we have
\[
I(c) = (-1)^{N(c)} q^{\frac{k(c)-N(c)}2} (1-a)(1-t)\; \widetilde{\HHH}(c).
\]
\end{thm}
\begin{proof}
The base of induction is the case where $c$ is a single short interval, corresponding to an unknot with $K_1$ attached to it. We have $\widetilde{\HHH}(c)=\frac{1}{1-t}$, $N(c)=k(c)=1$, so the right hand side is $a-1$. The left hand side is also $I(d_+(1))=a-1$.

Next we analyse how the extra factors change the rules A)-BE). The sign change is introduced in rules A) and C). Counting powers of $q$ is a bit tedious, but in the end we obtain a perfect match with the recursions for $I(c)$ obtained by applying Proposition \ref{prop: simplification} to the recursions given in \cite{mellit2016toric}:
\begin{enumerate}
\item[A)] $I(c)=(a-q^{k(c')} I(c')$,
\item[C)] $I(c)=-q^{k(c')-1} I(c')$,
\item[D)] $I(c)=q^{k(c')-1} I(c')$,
\item[BE)] $I(c)=q^{-\frac12} I(c') + t I(c'')$.
\end{enumerate}
\end{proof}

In particular, for $c=c_{m,n}$ we obtain a direct proof of the following result that can also be obtained by combining Theorem \ref{thm: dyck paths} with the main result of \cite{mellit2016toric}, and with the results of \cite{gorsky2015refined}. As explained in \cite{gorsky2015refined}, this result implies a conjecture of Aganagic and Shakirov \cite{aganagic2011knot}. See also \cite{cherednik2013jones}, \cite{cherednik2014daha}.

\begin{cor}\label{cor:daha}
The triply graded homology invariant $\cP_{m,n}=\cP(c_{m,n})$ of the $(m,n)$-torus knot is given by
\[
 (-1)^n \left(a (qt)^{-\frac12}\right)^{\frac{(m-1)(n-1)}2}\frac{1}{1-t}\frac{P_{m,n}(1)[a-1]}{a-1},
\]
where $P_{m,n}$ is the corresponding generator of the elliptic Hall algebra, viewed as an operator acting on symmetric functions.
\end{cor}

\section{Concluding remarks}
It is clear that our method and in particular Theorem \ref{thm: dyck paths} can be generalized to torus links (i.e. $(m,n)\neq 1$ case) in a straightforward way, however one has to deal with infinite sums and use Lemma \ref{lem:reduction2} in place of Lemma \ref{lem:reduction1}. In the case of $(n,n)$ links this leads to the combinatorics described in \cite{elias2016computation} and further elaborated in \cite{wilson2016torus}. To obtain the corresponding generalization of Theorem \ref{thm: daha} and Corollary \ref{cor:daha}, we are planning to generalize the representation of the monoid of positive toric braids to the monoid of more general \emph{tangles} in the thickening of the punctured torus. Such a generalization requires a more technical study of skein modules and will be treated in a separate publication.

It would be interesting to try to extend these methods to arbitrary iterated cables, probably satisfying some positivity assumptions. This would probably prove the more general conjectures of \cite{cherednik2014daha}. One would have to deal with links on tori which are themselves embedded in a complicated way in $\RR^3$. Then the transformation rules would somehow move the strands between the tori by dropping loops from more 'deep' tori to the less 'deep' ones, instead of simply contracting them. In a private conversation Vivek Shende suggested that such an extension should be straightforward.

Note that Dyck paths are in bijection with the cells of the corresponding compactified Jacobian of the plane curve singularity (see \cite{gorsky2013compactified}). Our construction, on the other hand, induces a filtration on the Khovanov-Rozansky homology. Moreover, this filtration is also parametrized by Dyck paths. It would be interesting to give a more direct comparison of the two. For instance, let $(m',n')$ be the closest lattice point to the diagonal. This corresponds to the first application of rule BE), so the complex computing Khovanov-Rozansky homology becomes an extension of two complexes. One would expect that the compactified Jacobian has a decomposition into an open subspace and a closed subspace such that the cells corresponding to Dyck paths passing through $(m', n')$ belong to the open subspace, while the remaining ones belong to the closed subspace, or vice versa. This decomposition should induce a filtration on the cohomology of the compactified Jacobian, which should match the filtration on the Khovanov-Rozansky homology.

While this paper was in preparation, I learned that Matt Hogancamp independently made certain progress in some special cases of torus knots and links (see \cite{hogancamp2017khovanov}). It is quite possible that in nice special cases the combinatorics can be dealt with without using the torus projection, like it was done in the $(n,n)$ case in \cite{elias2016computation}.

I thank Eugene Gorsky, Andrei Negut, Alexei Oblomkov, Peter Samuelson and Vivek Shende for useful discussions on the ideas of this paper. In particular, Eugene Gorsky turned my attention to the problem and told me about his and others' conjectures, made many useful comments and pointed out some inaccuracies.

\printbibliography

\end{document}